\documentclass[10pt]{amsart}

\usepackage{graphicx}
\usepackage{amsmath,amsfonts,amssymb}

\newtheorem{theorem}{Theorem}[section]
\newtheorem{lemma}[theorem]{Lemma}
\newtheorem{proposition}[theorem]{Proposition}
\newtheorem{corollary}[theorem]{Corollary}
\newtheorem{conjecture}[theorem]{Conjecture}

\theoremstyle{definition}
\newtheorem{definition}[theorem]{Definition}

\theoremstyle{remark}

\numberwithin{equation}{section}

\newcommand{\inv}{\ensuremath\mathrm{inv}}
\newcommand{\maj}{\ensuremath\mathrm{maj}}
\newcommand{\iDes}{\ensuremath\mathrm{iDes}}

\newlength\cellsize 

\savebox2{%
\begin{picture}(12,12)
\put(0,0){\line(1,0){12}}
\put(0,0){\line(0,1){12}}
\put(12,0){\line(0,1){12}}
\put(0,12){\line(1,0){12}}
\end{picture}}

\newcommand\cellify[1]{\def\thearg{#1}\def\nothing{}%
\ifx\thearg\nothing\vrule width0pt height\cellsize depth0pt%
  \else\hbox to 0pt{\usebox2\hss}\fi%
  \vbox to 12\unitlength{\vss\hbox to 12\unitlength{\hss$#1$\hss}\vss}}

\newcommand\tableau[1]{\vtop{\let\\=\cr
\setlength\baselineskip{-12000pt}
\setlength\lineskiplimit{12000pt}
\setlength\lineskip{0pt}
\halign{&\cellify{##}\cr#1\crcr}}}

\begin{document}


\title[Schur expansion of Macdonald polynomials]{Toward the Schur expansion of Macdonald polynomials}  

\author[S. Assaf]{Sami Assaf}
\address{Department of Mathematics, University of Southern California, Los Angeles, CA 90089}
\email{shassaf@usc.edu}





\begin{abstract}
  We give an explicit combinatorial formula for the Schur expansion of Macdonald polynomials indexed by partitions with second part at most two. This gives a uniform formula for both hook and two column partitions. The proof comes as a corollary to the result that generalized dual equivalence classes of permutations are unions of standard dual equivalence classes of permutations for certain cases, establishing an earlier conjecture of the author, and suggesting that this result can be generalized to arbitrary partitions.
\end{abstract}

\maketitle

%
\section{Introduction}
%
\label{sec:introduction}

The transformed Macdonald polynomials, $\widetilde{H}_{\mu}(X;q,t)$, a transformation of the polynomials introduced by Macdonald \cite{Mac88} in 1988, are the simultaneous generalization of Hall--Littlewood and Jack symmetric functions with two parameters, $q$ and $t$. The \emph{Kostka-Macdonald coefficients}, denoted $\widetilde{K}_{\lambda,\mu}(q,t)$, give the change of basis from Macdonald polynomials to Schur functions, namely,
\begin{displaymath}
  \widetilde{H}_{\mu}(X;q,t) = \sum_{\lambda} \widetilde{K}_{\lambda,\mu}(q,t) s_{\lambda}(X) .
\end{displaymath}
A priori, $\widetilde{K}_{\lambda,\mu}(q,t)$ is a rational function in $q$ and $t$ with rational coefficients.

The Macdonald Positivity Theorem \cite{Hai01}, first conjectured by Macdonald \cite{Mac88}, states that $\widetilde{K}_{\lambda,\mu} (q,t)$ is in fact a polynomial in $q$ and $t$ with nonnegative integer coefficients. 
Garsia and Haiman \cite{GH93} conjectured that the transformed Macdonald polynomials $\widetilde{H}_{\mu}(X;q,t)$ could be realized as the bi-graded characters of certain modules for the diagonal action of the symmetric group $S_n$ on two sets of variables. Once resolved, this conjecture gives a representation theoretic interpretation of Kostka-Macdonald coefficients as the graded multiplicity of an irreducible representation in the Garsia-Haiman module, and hence $\widetilde{K}_{\lambda,\mu} (q,t) \in \mathbb{N}[q,t]$. Following an idea outlined by Procesi, Haiman \cite{Hai01} proved this conjecture by analyzing the algebraic geometry of the isospectral Hilbert scheme of $n$ points in the plane, thereby establishing Macdonald Positivity. This proof, however, is purely geometric and does not offer a combinatorial interpretation for $\widetilde{K}_{\lambda,\mu} (q,t)$.

In 2004, Haglund \cite{Hag04} conjectured a combinatorial formula for the monomial expansion of $\widetilde{H}_{\mu}(X;q,t)$ that was subsequently prove by Haglund, Haiman and Loehr \cite{HHL05}. This formula establishes that $\widetilde{K}_{\lambda,\mu} (q,t) \in \mathbb{Z}[q,t]$ but comes short of proving non-negativity.

Combinatorial formulas for $\widetilde{K}_{\lambda,\mu}(q,t)$ have been found for certain special cases.  In 1995, Fishel \cite{Fis95} gave the first combinatorial interpretation for $\widetilde{K}_{\lambda,\mu}(q,t)$ when $\mu$ is a partition with $2$ columns, and there are now other formulas for two column Macdonald polynomials \cite{Zab99,LM03,Hag04,Ass08}. In all cases, finding broad extensions for these formulas has proven elusive. Haglund \cite{Hag04} conjectured a formula for three columns Macdonald polynomials, and this was later proved by Blasiak \cite{Bla16} who noted that his methods would not extend beyond this case. Both Roberts \cite{Rob14} and Loehr \cite{Loe17} noticed that Haglund's formula for two column Macdonald polynomials \cite{Hag04} extends to partitions with one additional box in the bottom row by observing that the proof using dual equivalence \cite{Ass08,Ass15} works for this case as well.

In this paper, we extend the combinatorial formula for two column Macdonald polynomials to partitions with second part at most $2$. This case simultaneously contains two column and hook partitions. The proof is purely combinatorial and combines the bijective proofs of the two column and single row Macdonald polynomials in \cite{Ass08} utilizing the structure of dual equivalence classes in \cite{Ass15}. In particular, it establishes \cite{Ass15}(Conjecture 5.6) for $\mu_2 \leq 2$.

\begin{figure}[ht]
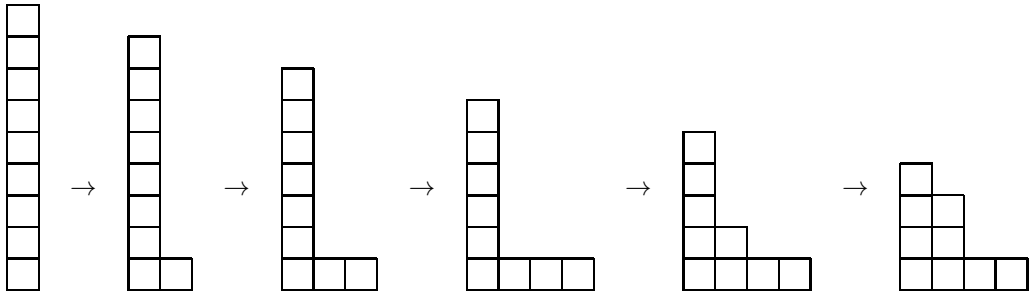

  \begin{displaymath}
    \tableau{ \ \\ \ \\ \ \\ \ \\ \ \\ \ \\ \ \\ \ \\ \ }
    \hspace{\cellsize} \raisebox{-5\cellsize}{$\rightarrow$} \hspace{\cellsize}
    \tableau{\\ \ \\ \ \\ \ \\ \ \\ \ \\ \ \\ \ \\ \ & \ }
    \hspace{\cellsize} \raisebox{-5\cellsize}{$\rightarrow$} \hspace{\cellsize}
    \tableau{\\ \\ \ \\ \ \\ \ \\ \ \\ \ \\ \ \\ \ & \ & \ }
    \hspace{\cellsize} \raisebox{-5\cellsize}{$\rightarrow$} \hspace{\cellsize}
    \tableau{\\ \\ \\ \ \\ \ \\ \ \\ \ \\ \ \\ \ & \ & \ & \ }
    \hspace{\cellsize} \raisebox{-5\cellsize}{$\rightarrow$} \hspace{\cellsize}
    \tableau{\\ \\ \\ \\ \ \\ \ \\ \ \\ \ & \ \\ \ & \ & \ & \ }
    \hspace{\cellsize} \raisebox{-5\cellsize}{$\rightarrow$} \hspace{\cellsize}
    \tableau{\\ \\ \\ \\ \\ \ \\ \ & \ \\ \ & \ \\ \ & \ & \ & \ }
  \end{displaymath}
  \caption{\label{fig:fold}Folding $(1^9)$ to $(4,2,2,1)$.}
\end{figure}

The main idea of this paper is to relate the standard dual equivalence classes for permutations, which coincide with the generalized dual equivalence classes for single column Macdonald polynomials, with the generalized dual equivalence classes for an arbitrary partition $\mu$. We do this incrementally, as indicated in Figure~\ref{fig:fold}, by folding the leg of the partition to form the rows, from bottom to top. In this paper, we succeed in defining bijections $\varphi_{\mu}$ for any partition $\mu$ with $\mu_2 \leq 2$ such that standard dual equivalence classes are combined into generalized $\mu$-dual equivalence classes. Thus we have a characterization of Kostka-Macdonald polynomials as 
\begin{displaymath}
  \widetilde{K}_{\lambda,\mu}(q,t) = \sum_{u \in \mathrm{SS}(\lambda)} q^{\inv_{\mu}(\varphi_{\mu}(u))} t^{\maj_{\mu}(\varphi_{\mu}(u))} ,
\end{displaymath}
where the sum is over certain permutations $u$ of type $\lambda$, and $\inv$ and $\maj$ are Haglund's statistics.

Our proofs are elementary and combinatorial, and this paper is largely self-contained. In Section~\ref{sec:mac}, we recall Haglund's permutation statistics from \cite{Hag04} and use them to define Macdonald polynomials combinatorially as done in \cite{HHL05}. In Section~\ref{sec:deg}, we recall the basic ideas of dual equivalence in \cite{Ass07,Ass15} and show how this machinery can be used to prove Macdonald positivity. In Section~\ref{sec:foata}, we recall Foata's bijection on permutations \cite{Foa68} and use it to prove our formula for hook partitions. In Section~\ref{sec:schur}, we recall a bijection from \cite{Ass08} and use it to prove our formula for partitions with second part at most $2$. 

%

%
\section{Macdonald polynomials}
%
\label{sec:mac}

We begin with terminology, notation and conventions. A \emph{permutation of $n$} is an ordering on the numbers $12\cdots n$. Though these are elements of the symmetric group $\mathfrak{S}_n$, we will not make use of the group structure. A \emph{word}, for our purposes, is a subword of a permutation, i.e., has no repeated letters. We regard $n$ as fixed throughout. Finally, let $X$ denote the infinite variable set $x_1,x_2,\ldots$.

Given a permutation $w$ and an index $i<n$, $i$ is an \emph{inverse descent of $w$} if $i+1$ lies to the left of $i$ in $w$. Denote the \emph{inverse descent set of $w$} by $\iDes(w)$. For example,
\[ \iDes(583691724) = \{2,4,7\} \]
For $w$ a permutation of $n$, $\iDes(w)$ is a subset of $[n-1] = \{1,2,\ldots,n-1\}$.

Gessel \cite{Ges84} defined an important basis of quasisymmetric functions indexed by subsets.

\begin{definition}[\cite{Ges84}]
  The \emph{fundamental quasisymmetric function} for $D \subseteq [n-1]$ is
  \begin{equation}
    F_{D}(X) = \sum_{\substack{i_1 \leq \cdots \leq i_n \\ j\in D \Rightarrow i_j < i_{j+1}}} x_{i_1} \cdots x_{i_n} .
  \end{equation}
  \label{def:gessel}
\end{definition}

For example, $F_{\{2,4,7\}}(X)$ contains the monomial $x_1^2x_3x_4x_5x_6^2x_8^2$ but not $x_1^2x_3x_4x_5x_6^3x_8$. 

Inverse descent sets allow us to associate a quasisymmetric function to each permutation. For example, the permutation $583691724$ will have the function $F_{\{2,4,7\}}(X)$ associated to it. An important application of this is the Frobenius character $H_n$  of the regular representation of $\mathfrak{S}_n$. As a function, $H_n$ may be written as the quasisymmetric generating function of permutations, 
\begin{equation}
  H_n(X) = \sum_{w \in \mathfrak{S}_n} F_{\iDes(w)}(X) .
\end{equation}

A \emph{partition of $n$} is a weakly decreasing sequence of positive integers that sum to $n$. Given a partition $\lambda$, a \emph{standard Young tableau of shape $\lambda$} is a permutation $w$ such that when the letters of $w$ fill the cells of the Young diagram of $\lambda$ from left to right, top to bottom, the rows increase left to right and the columns increase bottom to top. For example, see Figure~\ref{fig:SYT}.

\begin{figure}[ht]
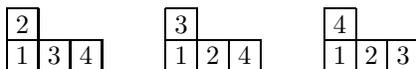

    \begin{displaymath}
      \tableau{2 \\ 1 & 3 & 4}
      \hspace{2\cellsize}
      \tableau{3 \\ 1 & 2 & 4}
      \hspace{2\cellsize}
      \tableau{4 \\ 1 & 2 & 3} 
    \end{displaymath}
    \caption{\label{fig:SYT}The standard Young tableaux of shape $(3,1)$.}
\end{figure}

Gessel \cite{Ges84} proved that summing over the reading words of all standard Young tableaux of a given shape precisely gives the Schur function indexed by that shape. We take this as our definition.

\begin{definition}[\cite{Ges84}]
  The \emph{Schur function} for $\lambda$ is
  \begin{equation}
    s_{\lambda}(X) = \sum_{T \in \mathrm{SYT}(\lambda)} F_{\iDes(w(T))}(X),
  \end{equation}
  where $\mathrm{SYT}(\lambda)$ is the set of standard Young tableaux of shape $\lambda$.
  \label{def:schur}
\end{definition}

Every character for $\mathfrak{S}_n$ has a unique decomposition as a sum of irreducible characters which coincide with Schur functions. The regular representation decomposes as
\begin{equation}
  H_n(X) = \sum_{\lambda} f_{\lambda} s_{\lambda}(X),
\end{equation}
where $f_{\lambda}$ is the number of standard Young tableaux of shape $\lambda$.

Macdonald \cite{Mac88} introduced polynomials that are a two-parameter analog of $H_n$. Haglund \cite{Hag04} defined new permutation statistics based on a partition shape that interpolate between the well-known major index and inversion number. Haglund, Haiman, and Loehr \cite{HHL05} proved that the generalized $\maj_{\mu}$ and $\inv_{\mu}$ statistics precisely give the Macdonald polynomials.

Given a permutation $w$ and a partition $\mu$, we fill the Young diagram of $\mu$ with $w$ from left to right, top to bottom. For example, see Figure~\ref{fig:mu}.

\begin{figure}[ht]
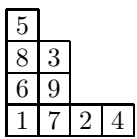

  \begin{displaymath}
    \tableau{5 \\ 8 & 3 \\ 6 & 9 \\ 1 & 7 & 2 & 4}
  \end{displaymath}
  \caption{\label{fig:mu}The Young diagram of $(4,2,2,1)$ filled with the permutation $583691724$.}
\end{figure}

\begin{definition}[\cite{Hag04}]
  Given a partition $\mu$, define the \emph{$\mu$-descent set of $w$}, denoted by $\mathrm{Des}_{\mu}(w)$, to be the set of cells of $\mu$ such that the $w$-entry is strictly greater than the $w$-entry immediately below in $\mu$. Define the \emph{$\mu$-major index of $w$}, denoted by $\maj_{\mu}(w)$, by
  \begin{equation}
    \maj_{\mu}(w) = \sum_{c\in\mathrm{Des}_{\mu}(w)} \mathrm{leg}(c)+1 ,
  \end{equation}
  where $\mathrm{leg}(c)$ is the number of cells strictly above $c$ in $\mu$.
  \label{def:maj}
\end{definition}

For example, from Figure~\ref{fig:mu}, we see that the $(4,2,2,1)$-descent set of $583691724$ consists of the cells containing $8,6,9$. Therefore the $(4,2,2,1)$-major index of $583691724$ is $2+3+2=7$. Notice that the $(1^n)$-major index is the usual major index statistic on permutations.

\begin{definition}[\cite{Hag04}]
  Given a partition $\mu$, define the \emph{$\mu$-inversion set of $w$}, denoted by $\mathrm{Inv}_{\mu}(w)$, to be the set of pairs $(w_i>w_j)$ with $i<j$ and either $w_i$ and $w_j$ are in the same row, or $w_j$ is one row lower and strictly left. Define the \emph{$\mu$-inversion number of $w$}, denoted by $\inv_{\mu}(w)$, by
  \begin{equation}
    \inv_{\mu}(w) = |\mathrm{Inv}_{\mu}(w)| - \sum_{c\in\mathrm{Des}_{\mu}(w)} \mathrm{arm}(c) ,
    \label{e:inv}
  \end{equation}
  where $\mathrm{arm}(c)$ is the number of cells strictly right of $c$ in $\mu$.
  \label{def:inv}
\end{definition}

For the example, we see that $583691724$ has $(4,2,2,1)$-inversion set $(8,3), (9,1), (7,2), (7,4)$, and the previously mentioned $(4,2,2,1)$-descent set of $8,6,9$. Therefore the $(4,2,2,1)$-inversion number is $4-1-1-0=2$. Notice that the $(n)$-inversion number is the usual inversion number.

As with Schur functions, we take a theorem as our definition for Macdonald polynomials. 

\begin{definition}[\cite{HHL05}]
  The \emph{Macdonald polynomial} for $\mu$ is
  \begin{equation}
    \widetilde{H}_{\mu}(X;q,t) = \sum_{w \in \mathfrak{S}_n} q^{\inv_{\mu}(w)} t^{\maj_{\mu}(w)} F_{\iDes(w)}(X) .
  \end{equation}
  \label{def:macdonald}
\end{definition}

To summarize the running example, the permutation $583691724$ will contribute $q^2 t^7 F_{\{2,4,7\}}(X)$ to the Macdonald polynomial $\widetilde{H}_{(4,2,2,1)}(X;q,t)$.



It remains an important open problem to prove, combinatorially, that the Schur coefficients of $\widetilde{H}_{\mu}(X;q,t)$ are nonnegative and to give an explicit combinatorial formula for the Schur expansion. We resolve this problem for the case when $\mu_2 \leq 2$. Precisely, we give a characterization of a set of \emph{super-standard} words together with a simple bijection $\varphi$ on permutations such that 
\begin{equation}
  \widetilde{H}_{\mu}(X;q,t) = \sum_{\lambda} \left( \sum_{u \in \mathrm{SS}(\lambda)} q^{\inv_{\mu}(\varphi(u))} t^{\maj_{\mu}(\varphi(u))} \right) s_{\lambda}(X),
\end{equation}
where $\mathrm{SS}(\lambda)$ is the set of super-standard words of type $\lambda$.

%
\section{Dual equivalence classes}
%
\label{sec:deg}

Dual equivalence involutions on permutations were studied in depth by Haiman \cite{Hai92}, though they also appear in earlier work of Edelman and Greene \cite{EG87}. We recall the basic definition here, along with a variation introduced by Assaf \cite{Ass07}.

\begin{definition}[\cite{Hai92}]
  Define the \emph{elementary dual equivalence involution} $d_i$, $1<i<n$, on permutations $w$ as follows. If $i$ lies between $i-1$ and $i+1$ in $w$, then $d_i(w)=w$. Otherwise, $d_i$ interchanges $i$ and whichever of $i\pm 1$ is further away from $i$.
  \label{def:switch}
\end{definition}

Since elementary dual equivalence involutions exchange a pair of consecutive values that are never adjacent, they clearly preserve the descent set of a permutation. In particular, they partition permutations of a given size and major index into equivalence classes. For example, Figure~\ref{fig:maj} shows the dual equivalence classes for permutations of $4$ with major index $2$.

\begin{figure}[ht]
  \begin{displaymath}
    \{ 2314 \stackrel{d_2}{\longleftrightarrow} 1324 \stackrel{d_3}{\longleftrightarrow} 1423 \}
    \hspace{2\cellsize}
    \{ 2413 \begin{array}{c}
      \stackrel{d_2}{\longleftrightarrow} \\[-1ex]
      \stackrel{\displaystyle\longleftrightarrow}{_{d_3}}
    \end{array} 3412 \} 
  \end{displaymath}
  \caption{\label{fig:maj}The dual equivalence classes of $w\in\mathfrak{S}_4$ with $\maj(w)=2$.}
\end{figure}

Define the \emph{de-standardization of $w$}, denoted by $\mathrm{dst}(w)$, to be the word obtained by changing $1,\ldots,i_1$ to $1$, $i_1+1,\ldots,i_1+i_2$ to $2$, and so on, where $\iDes(w) = \{i_1,i_2,\ldots\}$. For example,
\[ \mathrm{dst}(583691724) = 342341312. \]
The \emph{weight} of a de-standardization is the composition whose $i$th part is the number of occurrences of the letter $i$. For the example above, the weight of $\mathrm{dst}(583691724)$ is $(2,2,3,2)$.

A permutation $w$ is \emph{super-standard} if for every $k=n,\ldots,1$, the $\mathrm{dst}(w)_k\cdots \mathrm{dst}(w)_n$ has at least as many $i-1$s as $i$s for every $i$. That is, $w$ is super-standard if every suffix of $\mathrm{dst}(w)$ has partition weight. In this case, we say that $w$ has weight $\lambda$ whenever $\mathrm{dst}(w)$ has weight $\lambda$. For example,
\[ \mathrm{dst}(719852364) = 314321121, \]
which satisfies the suffix property. Therefore $719852364$ is super-standard of weight $(4,2,2,1)$.

\begin{theorem}
  Every permutation is dual equivalent to a unique super-standard permutation. Moreover, the quasisymmetric generating function of a dual equivalence class of a super-standard permutation is equal to the Schur function indexed by its weight.
  \label{thm:maj}
\end{theorem}

\begin{proof}
  The first statement follows from basic properties of the Robinson-Schensted insertion algorithm along with the characterization that $u$ is dual equivalent to $v$ if and only if $u$ and $v$ have the same recording tableau proved in \cite{Hai92}. The super-standard permutations are precisely those whose recording tableau is filled from bottom to top, left to right with the identity. That each class is a single Schur function now follows from Definition~\ref{def:schur}, and the characterization of which Schur function from the characterization of dominant elements in \cite{Ass15}(Theorem 4.7).
\end{proof}

For example, the quasisymmetric generating function of the dual equivalence classes in Figure~\ref{fig:maj} are $s_{(3,1)}(X)$ and $s_{(2,2)}(X)$, respectively, with super-standard representatives $1423$ of type $(3,1)$ and $3412$ of type $(2,2)$. In particular, Theorem~\ref{thm:maj} establishes the Schur positivity of $H_n$ directly from the combinatorial definition and provides an explicit rule for computing the Schur expansion. 

\begin{corollary}
  Let $H_n$ denote the Frobenius character of the regular representation of $\mathfrak{S}_n$. Then
  \begin{equation}
    H_n(X) = \sum_{\lambda} \left(\# \mathrm{SS}(\lambda) \right) s_{\lambda}(X),
  \end{equation}
  where $\mathrm{SS}(\lambda)$ denotes the set of super-standard permutations of weight $\lambda$.
  \label{cor:maj}
\end{corollary}
  
Assaf \cite{Ass07,Ass15} introduced a variation of the dual equivalence involutions that arise naturally from considering Haglund's permutation statistics for Macdonald polynomials.

\begin{definition}[\cite{Ass07,Ass15}]
  Define the \emph{elementary twisted dual equivalence involution} $\widetilde{d}_i$, $1<i<n$, on permutations $w$ as follows. If $i$ lies between $i-1$ and $i+1$ in $w$, then $\widetilde{d}_i(w)=w$. Otherwise, $\widetilde{d}_i$ cyclically rotates $i-1,i,i+1$ so that $i$ lies on the other side of $i-1$ and $i+1$.
  \label{def:twisted}
\end{definition}

Notice that elementary twisted dual equivalence involutions preserve the number of inversions, though not the set of inversion pairs. In particular, they partition permutations of a given size and inversion number into equivalence classes. For example, Figure~\ref{fig:inv} shows the twisted dual equivalence classes for permutations of $4$ with $2$ inversions.

\begin{figure}[ht]
  \begin{displaymath}
    \{ 2314 \stackrel{\widetilde{d}_2}{\longleftrightarrow} 3124
    \stackrel{\widetilde{d}_3}{\longleftrightarrow} 2143 
    \stackrel{\widetilde{d}_2}{\longleftrightarrow} 1342 
    \stackrel{\widetilde{d}_3}{\longleftrightarrow} 1423 \} 
  \end{displaymath}
  \caption{\label{fig:inv}The twisted dual equivalence class of $w\in\mathfrak{S}_4$ with $\inv(w)=2$.}
\end{figure}

For permutations $u,v$ of length $n$, $u$ is twisted dual equivalent to $v$ if and only if $\inv(u) = \inv(v)$ and $u_1>u_n$ if and only if $v_1>v_n$. Using this observation and properties of Foata's bijection \cite{Foa68} described in the following section, we have the following.

\begin{theorem}[\cite{Ass08}]
  The quasisymmetric generating function of a twisted dual equivalence class of permutations is symmetric and Schur positive.
  \label{thm:inv}
\end{theorem}

For example, the quasisymmetric generating function of the twisted dual equivalence class in Figure~\ref{fig:inv} is $s_{(3,1)}(X) + s_{(2,2)}(X)$, which is precisely the sum of the generating functions of the dual equivalence classes in Figure~\ref{fig:maj}. While the larger class in Figure~\ref{fig:inv} is not the union of the elements of the smaller classes in Figure~\ref{fig:maj}, we will show that it is after applying a suitable bijection.

Haglund's statistics interpolate between major index and inversion number. Analogously, the following involutions interpolate between dual equivalence and twisted dual equivalence. 

\begin{definition}[\cite{Ass07,Ass15}]
  Define involutions $D_i^{\mu}$, $1<i<n$, on permutations by
  \begin{equation}
    D_i^{\mu} (w) = \left\{ \begin{array}{rl}
      d_i(w) & \mbox{if $i$ is not a potential $\mu$-descent or $\mu$-inversion with $i\pm 1$}, \\
      \widetilde{d}_i(w) & \mbox{otherwise}. 
    \end{array} \right. 
  \end{equation}
  \label{def:Di}
\end{definition}

For example, both pairs of tableaux in Figure~\ref{fig:Dmu} are related by $D_7^{(4,2,2,1)}$. 

\begin{figure}[ht]
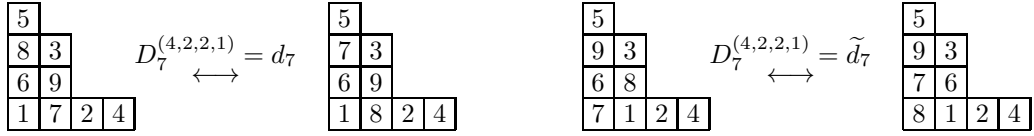

  \begin{displaymath}
    \tableau{5 \\ 8 & 3 \\ 6 & 9 \\ 1 & 7 & 2 & 4}
    \raisebox{-1.5\cellsize}{$\stackrel{\displaystyle D_7^{(4,2,2,1)} = d_7}{\longleftrightarrow}$} \hspace{\cellsize}
    \tableau{5 \\ 7 & 3 \\ 6 & 9 \\ 1 & 8 & 2 & 4}
    \hspace{4\cellsize}
    \tableau{5 \\ 9 & 3 \\ 6 & 8 \\ 7 & 1 & 2 & 4}
    \raisebox{-1.5\cellsize}{$\stackrel{\displaystyle D_7^{(4,2,2,1)} = \widetilde{d}_7}{\longleftrightarrow}$} \hspace{\cellsize}
    \tableau{5 \\ 9 & 3 \\ 7 & 6 \\ 8 & 1 & 2 & 4}
  \end{displaymath}
  \caption{\label{fig:Dmu}Two $(4,2,2,1)$-dual equivalence moves for $\{6,7,8\}$.}
\end{figure}

Haglund's formula relates Macdonald polynomials to LLT polynomials \cite{LLT97}, and so it follows from \cite{Ass15}(Proposition 5.2) that these involutions preserve Haglund's statistics. In particular, these involutions partition permutations of a given $\inv_{\mu}$ and $\maj_{\mu}$ statistic into equivalence classes. The main motivation for these involutions is the following Conjecture, which is a reformulation of \cite{Ass15}(Conjecture 5.6) for Macdonald polynomials.

\begin{conjecture}[\cite{Ass15}]
  For $\mu$ a partition, the quasisymmetric generating function of each generalized dual equivalence class under $D_i^{\mu}$ is symmetric and Schur positive.
  \label{conj:class}
\end{conjecture}

Since generalized dual equivalence classes under $D_i^{\mu}$ have constant $\inv_{\mu}$ and $\maj_{\mu}$ statistics, Conjecture~\ref{conj:class} is enough to establish Macdonald Positivity. In \cite{Ass15}, the conjecture is shown to hold for $\mu$ a two column partition and for $\mu$ a single row. In this paper, we merge these two cases to establish Conjecture~\ref{conj:class} for partitions $\mu$ with $\mu_2 \leq 2$. The proof utilizes explicit bijections on words that merge equivalence classes as the shape $\mu$ is transformed from a single column.

%
\section{Foata's bijection and hooks}
%
\label{sec:foata}

Foata \cite{Foa68} constructed a bijection on words with the property that the major index of a word equals the inversion number of its image, thereby providing a bijective proof of the equi-distribution of these two statistics. Moreover, his bijection preserves the inverse descent set, and so it proves that the quasisymmetric generating function of permutations with major index $k$ is equal to the quasisymmetric generating function of permutations with $k$ inversions. 

Foata's bijection makes recursive use of a family of bijections $\gamma_x$ indexed by a letter $x$. Given a word $w$ and a letter $x$ not in $w$, define a partitioning $\Gamma_x(w)$ of $w$ by: if $w_1<x$, then break before each index $i$ such that $w_i < x$; otherwise break before each index $i$ such that $w_i>x$. For example,
\[ \Gamma_5(83691724) = \ \mid 83 \mid 6 \mid 91 \mid 724 . \]
The bijection $\gamma_x$ is defined by cycling the first letter of each block of $\Gamma_x$ to the end of the block. Continuing with the example,
\[ \gamma_5(83691724) = 38619247 . \]

\begin{proposition}
  For $u,v$ words and $x$ a letter, we have $\iDes(uxv) = \iDes(ux\gamma_x(v))$.
  \label{prop:gamma-iDes}
\end{proposition}

\begin{proof}
  We claim that the relative order of consecutive letters is the same in $uxv$ as in $ux\gamma_x(v)$. If they are not in the same block of $\Gamma_x(v)$, either by being in different blocks or by at least one being in $ux$, then since the relative order of $ux$ and each of the blocks is preserved, their relative order remains the same in $ux\gamma_x(v)$. If they are in the same block of $\Gamma_x(v)$, then they are both larger or both smaller than $x$, and so their relative order remains unchanged by $\gamma_x$. 
\end{proof}

Define a family of bijections $\phi_k$, for $k\in [n]$, on permutations by
\begin{equation}
  \phi_k(w) = w_1\cdots w_k \gamma_{w_k}(w_{k+1} \cdots w_n) .
\end{equation}
For example, $\phi_4(583691724) = 5836 \gamma_6(91724) = 583619247$. We use the map $\phi_k$ to relate $(n-k,1^k)$-dual equivalence classes with $(n-k+1,1^{k-1})$-dual equivalence classes as illustrated in Figure~\ref{fig:hook}. 

\begin{figure}[ht]
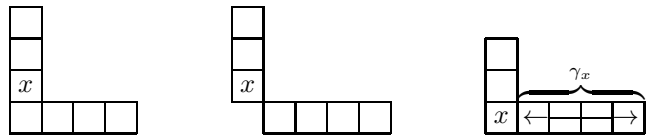

  \begin{displaymath}
    \tableau{ \ \\ \ \\ x \\ \ & \ & \ & \ } \hspace{3\cellsize}
    \tableau{ \ \\ \ \\ x \\ & \ & \ & \ & \ } \hspace{3\cellsize}
    \tableau{ \\ \ \\ \ \\ x & \hfill\raisebox{-.4\cellsize}{$\leftarrow$} & \line(1,0){12} & \line(1,0){12} & \raisebox{-.4\cellsize}{$\rightarrow$}\hfill } \hspace{-4\cellsize} \raisebox{-2\cellsize}{$\overbrace{\hspace{4\cellsize}}^{\gamma_x}$}
  \end{displaymath}
  \caption{\label{fig:hook}Sliding a hook to lengthen the arm.}
\end{figure}

To begin to relate dual equivalence classes with $\mu$-equivalence classes, we consider first the cases when the bijection $\phi_k$ commutes with the generalized dual equivalence involutions.

\begin{lemma}
  Let $\mu = (n-k,1^k)$ for some $k>0$, and let $\nu = (n-k+1,1^{k-1})$. If $D_i^{\mu}(u) = u$, $u_k \not\in\{i-1,i,i+1\}$, or $u_j \in\{i-1,i,i+1\}$ for some $j<k$, then $\phi_k(D_i^{\mu}(u)) = D_i^{\nu}(\phi_k(u))$.
  \label{lem:hooks}
\end{lemma}

\begin{proof}
  By definition, $D_i^{\mu}(w) = w$ if and only if both or neither of $i-1,i$ are in $\iDes(w)$. By Proposition~\ref{prop:gamma-iDes}, $\phi_k$ preserves $\iDes$, and so $D_i^{\mu}(u) = u$ if and only if $D_i^{\nu}(\phi_k(u)) = u$. 

  If $u_k\not\in\{i-1,i,i+1\}$, then $\gamma_{u_k}$ does not differentiate between $i-1,i,i+1$, and so their relative order is the same in $u$ and in $\phi_k(u)$. Moreover, since the number of $i-1,i,i+1$ in the leg is the same for $u$ and in $\phi_k(u)$, $D_i^{\mu}$ and $D_i^{\nu}$ act by the same involution, so $\phi_k(D_i^{\mu}(u)) = D_i^{\nu}(\phi_k(u))$.
  
  Finally, if $u_j \in\{i-1,i,i+1\}$ for some $j<k$, then $D_i^{\mu}(u) = d_i(u)$, and since $\phi_k$ leaves the leg unchanged, we also have $D_i^{\mu}(\phi_k(u)) = d_i(\phi_k(u))$. Since $d_i$ exchanges consecutive values, the middle letter of $i-1,i,i+1$ occurring in $u$ compares the same with both values being exchanged. Therefore $\Gamma_{u_k}$ and $\Gamma_{d_i(u)_k}$ partition the arm in the same way, and so $\phi_k(d_i(u)) = d_i(\phi_k(u))$ as desired.
\end{proof}

\begin{figure}[ht]
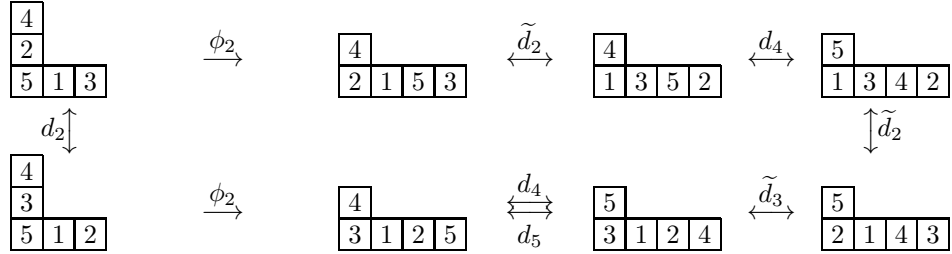

  \begin{displaymath}
    \begin{array}{c@{\hskip 3\cellsize}c@{\hskip 3\cellsize}ccccc}
      \tableau{4 \\ 2 \\ 5 & 1 & 3} &
      \raisebox{-\cellsize}{$\stackrel{\displaystyle\phi_2}{\longrightarrow}$} &
      \tableau{\\ 4 \\ 2 & 1 & 5 & 3} &
      \raisebox{-\cellsize}{$\stackrel{\displaystyle \widetilde{d}_2}{\longleftrightarrow}$} &
      \tableau{\\ 4 \\ 1 & 3 & 5 & 2} &
      \raisebox{-\cellsize}{$\stackrel{\displaystyle d_4}{\longleftrightarrow}$} &
      \tableau{\\ 5 \\ 1 & 3 & 4 & 2} \\[2\cellsize]
      d_2 \rotatebox[origin=c]{90}{$\longleftrightarrow$} & & & & & &
      \rotatebox[origin=c]{90}{$\longleftrightarrow$} \widetilde{d}_2  \\
      \tableau{4 \\ 3 \\ 5 & 1 & 2} &
      \raisebox{-\cellsize}{$\stackrel{\displaystyle\phi_2}{\longrightarrow}$} &
      \tableau{\\ 4 \\ 3 & 1 & 2 & 5} &
      \raisebox{-\cellsize}{$
        \begin{array}{c}
          \stackrel{\displaystyle d_4}{\longleftrightarrow} \\[-1ex]
          \stackrel{\displaystyle\longleftrightarrow}{d_5}
        \end{array} $} &
      \tableau{\\ 5 \\ 3 & 1 & 2 & 4} &
      \raisebox{-\cellsize}{$\stackrel{\displaystyle \widetilde{d}_3}{\longleftrightarrow}$} &
      \tableau{\\ 5 \\ 2 & 1 & 4 & 3}
    \end{array}
  \end{displaymath}
  \caption{\label{fig:noncommute}Illustration of the $(4,1)$-dual equivalence for $42153$ and $43125$.}
\end{figure}

Generalized dual equivalence involutions do not always commute with $\phi_k$. For example, take $\mu = (3,1,1)$ and consider $u=42513$. Then $\phi_2(u) = 42\gamma_2(513) = 42153$ and $\phi_2(D_2^{\mu}(u)) = 43\gamma_3(512) = 43125 \neq D^{(4,1)}_2(42153)$. They are, however, $(4,1)$-dual equivalent, as illustrated in Figure~\ref{fig:noncommute}. 

\begin{lemma}
  Let $u$ be a permutation with $u_{k+1}=3$, $u_j=1$, and $u_i=2$ for some $i>j>k+1$, and let $x = u_{i+1}$. Let $v$ be the permutation with $v_{k+1} = 1$, $v_j = 3$, $v_i = x$, $v_{i+1} = 2$, and $v_h = u_h$ for all other indices $h$. Then $u$ and $v$ are $(n-k,1^k)$-dual equivalent.
  \label{lem:inv}
\end{lemma}

\begin{figure}[ht]
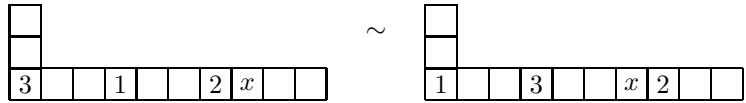

  \begin{displaymath}
    \tableau{ \ \\ \ \\ 3 & \ & \ & 1 & \ & \ & 2 & x & \ & \ }
    \hspace{\cellsize} \sim \hspace{\cellsize}
    \tableau{ \ \\ \ \\ 1 & \ & \ & 3 & \ & \ & x & 2 & \ & \ }
  \end{displaymath}
  \caption{\label{fig:lemma}Illustration of the situation described in Lemma~\ref{lem:inv}.}
\end{figure}

\begin{proof}
  If $x=4$, then $\widetilde{d}_3(3124) = 2143 = \widetilde{d}_2(1342)$, so $3124$ is $(4)$-equivalent to $1342$, and this holds whether other letters are interspersed or not. Therefore we may proceed by induction on $x$ assuming $x>4$. If only $1,2,3,x$ occur in the arm, then we may apply $d_{x-1}$ to exchange $x$ with $x-1$, use induction since $x-1$ is now in the arm, and then apply $d_{x-1}$ again to restore $x$. Therefore we may also assume there is some entry $3<a<x$ also in the arm.

  If $a$ lies left of $x$, then by induction and ignoring letters larger than $a$ (in particular, ignoring $x$), the pattern $13ax2$ may be exchanged for $312xa$. Applying $\widetilde{d}_2$ gives the pattern $231xa$. It is an easy exercise to show that the twisted Knuth move taking $1xa$ is generalized dual equivalent to $a1x$, and so applying $\widetilde{d}_2$ again yields the pattern $31a2x$, proving the result. Similarly, if $a$ lies right of $x$, then by induction and ignoring letters larger than $a$ (in particular, ignoring $x$), the pattern $312xa$ may be exchanged for $13ax2$. Ignoring the $1$ which lies precisely in the corner, using that the twisted Knuth move taking $ax1$ is generalized dual equivalent to $x1a$, we obtain $13x2a$, as desired.
\end{proof}

For example, we see from Figure~\ref{fig:noncommute} that $43125$ and $41352$ are $(4,1)$-equivalent.

\begin{theorem}
  Let $\mu = (m,1^k)$ be a hook partition for some $m,k>0$, and let $\nu = (m+1,1^{k-1})$. For permutations $u,v$, if $u$ and $v$ are $\mu$-equivalent, then $\phi_k(u)$ and $\phi_k(v)$ are $\nu$-equivalent.
  \label{thm:hooks}
\end{theorem}

\begin{proof}
  It is enough to prove the result for $u$ and $v = D_i^{\mu}(u)$. By Lemma~\ref{lem:hooks}, we may assume that $D_i^{\mu}(u)\neq u$, $u_k \in \{i-1,i,i+1\}$, and the remaining two of $i-1,i,i+1$ are among $u_{k+1} \cdots u_{k+m}$. In particular, this implies $D_i^{\mu}(u) = d_i(u)$ and $D_i^{\nu}(\phi_k(u)) = \widetilde{d}_i(\phi_k(u))$. By symmetry between $u$ and $v$, we may assume $u_k=i$. We treat in detail the case where $i-1$ is left of $i+1$, noting that the alternative case is completely analogous.

  We compare $\Gamma_i(u_{k+1} \cdots u_n)$ with $\Gamma_{i+1}(v_{k+1} \cdots v_n)$. Since $v = d_i(u)$, we must have $u_{k+1} = v_{k+1}$ and neither is $i$ or $i+1$. Suppose first that $u_{k+1}<i$. If $i-1$ and $i+1$ occur within the same $\Gamma_i$-block of $u$, then we have
  \begin{displaymath}
    \begin{array}{rclcr}
      \Gamma_i(u_{k+1} \cdots u_n) & = & ( \cdots \mid (i-1) A (i+1) B \mid \cdots )
      & \stackrel{\displaystyle\gamma_{i}}{\longmapsto} & i \cdots A (i+1) B (i-1) \cdots \\
      \Gamma_{i+1}(v_{k+1} \cdots v_n) & = & ( \cdots \mid (i-1) A \mid (\makebox[1.2em]{$i$}) B \mid \cdots )
      & \stackrel{\displaystyle\gamma_{i+1}}{\longmapsto} & i+1 \cdots A (i-1) B (\makebox[2em]{$i$}) \cdots 
    \end{array}
  \end{displaymath}
  where $A,B$ are words containing letters larger than $i+1$. Therefore $D_i^{\nu}(\phi_k(u)) = \widetilde{d}_i(\phi_k(u)) = \phi_k(v) = \phi_k(d_i(u)) = \phi_k(D_i^{\mu}(u))$. If $i-1$ and $i+1$ do not occur within the same $\Gamma_i$-block of $u$, then computing the $\Gamma$-partitioning and applying $\gamma$, we have
  \begin{displaymath}
    \begin{array}{rlcr}
      \Gamma_i : & 
      ( \cdots \mid (i-1) A \mid \cdots \mid x B (i+1) C \mid \cdots )
      & \stackrel{\displaystyle\gamma_{i}}{\longmapsto} &
      i \cdots A (i-1) \cdots B (i+1) C x \cdots \\
      \Gamma_{i+1} : &
      ( \cdots \mid (i-1) A \mid \cdots \mid x B \mid (\makebox[1.2em]{$i$}) C \mid \cdots )
      & \stackrel{\displaystyle\gamma_{i+1}}{\longmapsto} & 
      i+1 \cdots A (i-1) \cdots B x C (\makebox[2em]{$i$}) \cdots
    \end{array}
  \end{displaymath}
  where $A,B,C$ are words containing letters larger than $i+1$, and $x<i-1$ is a letter. If we consider $\phi_k(v)$ and $\widetilde{d}_i(\phi_k(u))$ where we delete all letters larger than $i+1$ and change smaller letters $a$ to $(i+1)-a+2$, we are precisely in the situation of Lemma~\ref{lem:inv}. Reversing this change and replacing larger letters does not affect the equivalence, and so $\phi_k(v)$ and $\widetilde{d}_i(\phi_k(u))$ are $\nu$-dual equivalent. Since $i-1,i,i+1$ all lie in the arm of $\nu$, we have $\phi_k(u)$ and $\widetilde{d}_i(\phi_k(u))$ are $\nu$-dual equivalent as well. Therefore, by transitivity, $\phi_k(u)$ and $\phi_k(v)$ are $\nu$-dual equivalent.

  For the case  $u_{k+1}>i+1$, we apply the same analysis. If $i-1$ and $i$ occur within the same $\Gamma_{i+1}$-block of $v$, then we see that $D_i^{\nu}(\phi_k(u)) = \widetilde{d}_i(\phi_k(u)) = \phi_k(v) = \phi_k(d_i(u)) = \phi_k(D_i^{\mu}(u))$, giving the desired $\nu$-dual equivalence. If $i-1$ and $i$ do not occur within the same $\Gamma_{i+1}$-block of $v$, then deleting letters smaller than $i-1$ and changing larger letters $a$ to $a-(i-1)+1$ puts $\phi_k(v)$ and $\widetilde{d}_i(\phi_k(u))$ precisely in the situation of Lemma~\ref{lem:inv}, and so $\phi_k(v)$ and $\widetilde{d}_i(\phi_k(u))$ are $\nu$-dual equivalent. Again, since $i-1,i,i+1$ all lie in the arm of $\nu$, by transitivity  $\phi_k(u)$ and $\phi_k(v)$ are $\nu$-dual equivalent.
\end{proof}

Since, by Proposition~\ref{prop:gamma-iDes}, $\phi_k$ preserves the inverse descent set, Theorem~\ref{thm:hooks} establishes Conjecture~\ref{conj:class} for hooks and gives the following explicit formula.

\begin{corollary}
  For $\mu = (n-k,1^k)$ a hook partition, set $\varphi_{\mu} = \phi_{k+1} \cdots \phi_{n-1}$. Then we have
  \begin{equation}
    \widetilde{H}_{\mu}(X;q,t) = \sum_{\lambda} \left( \sum_{u \in \mathrm{SS}(\lambda)} q^{\inv_{\mu}(\varphi_{\mu}(u))} t^{\maj_{\mu}(\varphi_{\mu}(u))} \right) s_{\lambda}(X) .
  \end{equation}
  \label{cor:hooks}
\end{corollary}

\begin{figure}[ht]
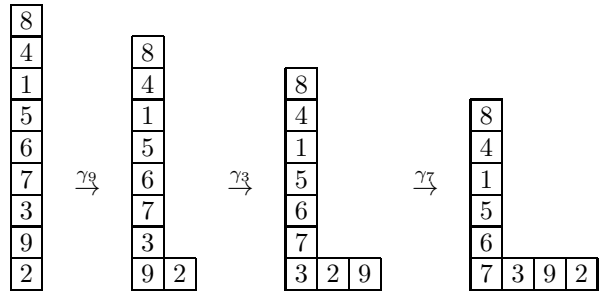

  \begin{displaymath}
    \tableau{ 8 \\ 4 \\ 1 \\ 5 \\ 6 \\ 7 \\ 3 \\ 9 \\ 2 }
    \hspace{\cellsize} \raisebox{-5\cellsize}{$\stackrel{\gamma_9}{\rightarrow}$} \hspace{\cellsize}
    \tableau{\\ 8 \\ 4 \\ 1 \\ 5 \\ 6 \\ 7 \\ 3 \\ 9 & 2 }
    \hspace{\cellsize} \raisebox{-5\cellsize}{$\stackrel{\gamma_3}{\rightarrow}$} \hspace{\cellsize}
    \tableau{\\ \\ 8 \\ 4 \\ 1 \\ 5 \\ 6 \\ 7 \\ 3 & 2 & 9 }
    \hspace{\cellsize} \raisebox{-5\cellsize}{$\stackrel{\gamma_7}{\rightarrow}$} \hspace{\cellsize}
    \tableau{\\ \\ \\ 8 \\ 4 \\ 1 \\ 5 \\ 6 \\ 7 & 3 & 9 & 2 }
  \end{displaymath}
  \caption{\label{fig:fold-hook}Folding $(1^9)$ to $(4,1^5)$.}
\end{figure}

See Figure~\ref{fig:fold-hook} for details of the map $\varphi_{(4,1^5)}$ on the permutation $841567392$. Notice that the inverse descent set for all four fillings is $\{2,3,7\}$. From left to right, the Macdonald weights are $t^{17}$, $q t^{9}$, $q t^{9}$, $q^4 t^{3}$. We emphasize that in Corollary~\ref{cor:hooks}, the weights change with the bijection, but the set of permutations that determines the Macdonald polynomial is the same for every partition.

%
\section{Folding the legs}
%
\label{sec:schur}

Refining Foata's bijection, Assaf \cite{Ass08} constructed a family of bijections on words that preserve an interpolating statistic between major index and inversion number. We focus on the first in this family of bijections, which makes use a bijection $\beta_x$ indexed by a letter $x$. The idea of $\beta_x$ is to swap adjacent pairs of letters, percolating through the word based on a straddling condition.

Given a word $w$ and a letter $x$ not in $w$, let $B_x$ be the set of indices defined recursively as follows: if $x$ has value between $w_1,w_2$, then $1 \in B_x$; if $i\in B_x$, then if exactly one of $w_i,w_{i+1}$ has value between $w_{i+2},w_{i+3}$, then $i+2 \in B_x$. For example,
\[ B_5(83691724) = \{1, 3, 5\}. \]
The bijection $\beta_x$ is defined by swapping $w_i$ and $w_{i+1}$ for every $i \in B_x$. Continuing with the example,
\[ \beta_5(83691724) = 38967124. \]

\begin{proposition}
  For $u,v,y$ words and $x$ a letter, we have $\iDes(uxvy) = \iDes(ux\beta_x(v)y)$.
  \label{prop:beta-iDes}
\end{proposition}

\begin{proof}
  Since $\beta_x$ exchanges pairs of adjacent letters that are necessarily not consecutive, the relative order of any pair of consecutive entries remains unchanged. 
\end{proof}

\begin{figure}[ht]
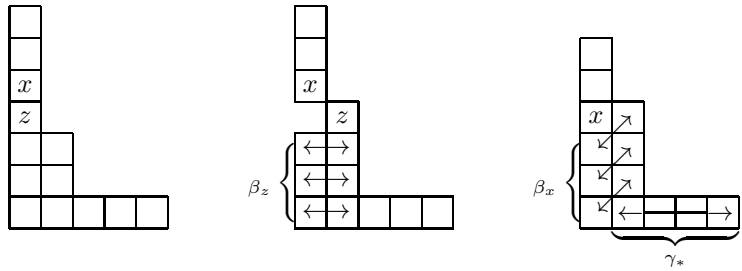

  \begin{displaymath}
    \tableau{ \ \\ \ \\ x \\ z \\ \ & \ \\ \ & \ \\ \ & \ & \ & \ & \ } \hspace{4\cellsize}
    \raisebox{-4.75\cellsize}{\makebox[0pt]{$_{\beta_z} \left\{ \tableau{\\ \\ } \right.$ }}
    \tableau{ \ \\ \ \\ x \\ & z \\
      \ & \makebox[0pt]{\hspace{-\cellsize}$\longleftrightarrow$} \\
      \ & \makebox[0pt]{\hspace{-\cellsize}$\longleftrightarrow$} \\
      \ & \makebox[0pt]{\hspace{-\cellsize}$\longleftrightarrow$} & \ & \ & \ } \hspace{4\cellsize}
    \raisebox{-4.75\cellsize}{\makebox[0pt]{$_{\beta_x} \left\{ \tableau{\\ \\ } \right.$ }}
    \tableau{ \\ \ \\ \ \\ x & \ \\
      \raisebox{1.5\cellsize}{\makebox[0pt]{\hspace{1.25\cellsize}\rotatebox[origin=c]{45}{$\longleftrightarrow$}}} & \ \\
      \raisebox{1.5\cellsize}{\makebox[0pt]{\hspace{1.25\cellsize}\rotatebox[origin=c]{45}{$\longleftrightarrow$}}} & \ \\
      \raisebox{1.5\cellsize}{\makebox[0pt]{\hspace{1.25\cellsize}\rotatebox[origin=c]{45}{$\longleftrightarrow$}}} & \hfill\raisebox{-.5\cellsize}{$\leftarrow$} & \line(1,0){12} & \line(1,0){12} & \raisebox{-.5\cellsize}{$\rightarrow$}\hfill }
    \hspace{-4\cellsize} \raisebox{-6\cellsize}{$\underbrace{\hspace{4\cellsize}}_{\gamma_*}$}
  \end{displaymath}
  \caption{\label{fig:tower}Sliding the leg and dropping the foot.}
\end{figure}

Figure~\ref{fig:tower} illustrates the manner in which we will use $\beta_x$ along with $\gamma_x$ to relate generalized dual equivalence classes for partitions with second part at most $2$. Define a family of bijections $\sigma_{k}$ by
\begin{equation}
  \sigma_{(k,m)} (w) = w_1 \cdots w_k \beta_{w_k}( w_{k+1} \cdots w_{k+m} ) w_{k+m+1} \cdots w_n .
\end{equation}

If we extend Definition~\ref{def:macdonald} to arbitrary diagrams, then we may consider the quasisymmetric generating function associated with the middle diagram in Figure~\ref{fig:tower} and its corresponding generalized dual equivalence classes. Extending \cite{Ass08}(Theorem~5.3), we have the following commutativity.

\begin{lemma}
  Let $\mu = (n-2b-a,2^b,1^a)$, and let $\delta$ be the diagram obtained by moving the $a$th cell from the top to the second column. Then $\sigma_{(a,2b+2)}(D_i^{\mu}(u)) = D_i^{\delta}(\sigma_{(a,a+2b+2)}(u))$.
  \label{lem:diagram}
\end{lemma}

\begin{proof}
  For any diagram $\lambda$ and any permutation $w$, $D_i^{\lambda}(w) = w$ if and only if both or neither of $i-1,i$ are inverse descents of $w$. By Proposition~\ref{prop:beta-iDes}, $\phi_k$ preserves the inverse descent set, and so $D_i^{\mu}(u) = u$ if and only if $D_i^{\delta}(\sigma_{(k,m)}(u)) = u$. Therefore assume that $D_i^{\mu}$ acts non-trivially on $u$.

  Let $z = u_a$, and consider the action of $\beta_{z}$ on $u_{a+1} \cdots u_{a+2b+2}$. From the definition, $\beta_z$ cannot exchange $i$ and $i\pm 1$ since no letter has value between these. Similarly, $\beta_{z}$ exchanges $i-1$ and $i+1$ only if they are the leftmost pair in the same row and $i$ is in the row above. When this is not the case, both $D_i^{\mu}$ and $D_i^{\delta}$ act by the same operator, either $d_i$ or $\widetilde{d}_i$, and the actions are easily seen to commute with $\sigma_{(a,2b+2)}$ since letters $j>i+1$ or $h<i-1$ compare the same with each of $i-1,i,i+1$. When this is the case, one of $D_i^{\mu}$ and $D_i^{\delta}$ acts by $d_i$ and the other by $\widetilde{d}_i$, since in one case the $i$ will be above the left letter and so the action is $d_i$ and in the other it will be above the right forcing the action by $\widetilde{d}_i$. Since the difference between these two actions is precisely exchanging $i-1$ and $i+1$, we again have $\sigma_{(a,2b+2)}(D_i^{\mu}(u)) = D_i^{\delta}(\sigma_{(a,a+2b+2)}(u))$.
\end{proof}

To compare generalized dual equivalences classes for $(n-2b-a, 2^b, 1^a)$ with those for $(n-2b-a, 2^{b+1}, 1^{a-2})$, we use the composite map
\begin{equation}
  \phi_{(a,b)}(w) = \left\{ \begin{array}{rl}
    \phi_{a+2b+1} \sigma_{(a-1,2b+2)} \sigma_{(a,2b+2)}(w)
    & \mbox{if $\sigma_{(a-1,2b+2)}$ changes $\sigma_{(a,2b+2)} (w)_{a+2b+1}$,} \\
    \sigma_{(a-1,2b+2)} \sigma_{(a,2b+2)} (w)
    & \mbox{otherwise.}
  \end{array} \right.
\end{equation}

\begin{theorem}
  Let $\mu = (n-2b-a,2^b,1^a)$, and let $\nu = (n-2b-a,2^{b+1},1^{a-2})$. For permutations $u,v$, if $u$ and $v$ are $\mu$-equivalent, then $\phi_{(a,b)}(u)$ and $\phi_{(a,b)}(v)$ are $\nu$-equivalent.
  \label{thm:tower}
\end{theorem}

\begin{proof}
  It is enough to prove the result for $u$ and $v = D_i^{\mu}(u)$. By Lemma~\ref{lem:diagram}, if $\sigma_{(a-1,2b+2)}$ does not alter $\sigma_{(a,2b+2)} (u)_{a+2b+1}$, then $\phi_{(a,b)}(D_i^{\mu}(u)) = D_i^{\nu}(\phi_{(a,b)}(u))$, and so $\phi_{(a,b)}(u)$ and $\phi_{(a,b)}(v)$ differ by an elementary $\nu$-equivalence. Similarly, by Lemma~\ref{lem:hooks}, if $\sigma_{(a,2b+2)} (u)_{a+2b+1}$ is not one of $i-1,i,i+1$ or if one of $i-1,i,i+1$ is not in the bottom row, then again $\phi_{(a,b)}(D_i^{\mu}(u)) = D_i^{\nu}(\phi_{(a,b)}(u))$, and so again $\phi_{(a,b)}(u)$ and $\phi_{(a,b)}(v)$ differ by an elementary $\nu$-equivalence. Finally, the analysis in the proof of Theorem~\ref{thm:hooks} resolves the case when $\sigma_{(a,2b+2)} (u)_{a+2b+1}$ is one of $i-1,i,i+1$ and the others lie in the bottom row. In this final case, $\phi_{(a,b)}(u)$ and $\phi_{(a,b)}(v)$ are $\nu$-equivalent though not necessarily by a single elementary $\nu$-equivalence.
\end{proof}

By Propositions~\ref{prop:gamma-iDes} and \ref{prop:beta-iDes}, $\phi_k$ and $\phi_{(a,b)}$ preserve the inverse descent set. Therefore Theorems~\ref{thm:hooks} and \ref{thm:tower} together establish Conjecture~\ref{conj:class} for partitions with second part at most $2$. Moreover, we have the following explicit formula for the Macdonald polynomial.

\begin{corollary}
  For $\mu = (n-2b-a,2^b,1^a)$, set $\varphi_{\mu} = \phi_{(a+2,b-2)} \cdots \phi_{(a+2b,0)} \phi_{a+2b+1} \cdots \phi_{n-1}$. Then
  \begin{equation}
    \widetilde{H}_{\mu}(X;q,t) = \sum_{\lambda} \left( \sum_{u \in \mathrm{SS}(\lambda)} q^{\inv_{\mu}(\varphi_{\mu}(u))} t^{\maj_{\mu}(\varphi_{\mu}(u))} \right) s_{\lambda}(X) .
  \end{equation}
  \label{cor:tower}
\end{corollary}

\begin{figure}[ht]
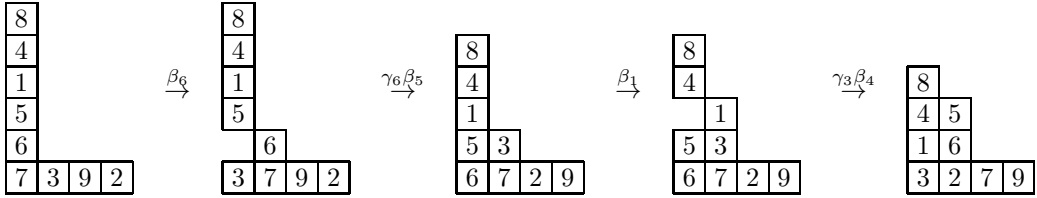

  \begin{displaymath}
    \tableau{ 8 \\ 4 \\ 1 \\ 5 \\ 6 \\ 7 & 3 & 9 & 2 }
    \hspace{\cellsize} \raisebox{-2\cellsize}{$\stackrel{\beta_6}{\rightarrow}$} \hspace{\cellsize}
    \tableau{ 8 \\ 4 \\ 1 \\ 5 \\ & 6 \\ 3 & 7 & 9 & 2 }
    \hspace{\cellsize} \raisebox{-2\cellsize}{$\stackrel{\gamma_6\beta_5}{\rightarrow}$} \hspace{\cellsize}
    \tableau{\\ 8 \\ 4 \\ 1 \\ 5 & 3 \\ 6 & 7 & 2 & 9 }
    \hspace{\cellsize} \raisebox{-2\cellsize}{$\stackrel{\beta_1}{\rightarrow}$} \hspace{\cellsize}
    \tableau{\\ 8 \\ 4 \\ & 1 \\ 5 & 3 \\ 6 & 7 & 2 & 9 }
    \hspace{\cellsize} \raisebox{-2\cellsize}{$\stackrel{\gamma_3\beta_4}{\rightarrow}$} \hspace{\cellsize}
    \tableau{\\ \\ 8 \\ 4 & 5 \\ 1 & 6 \\ 3 & 2 & 7 & 9 }
  \end{displaymath}
  \caption{\label{fig:fold-tower}Folding $(4,1^5)$ to $(4,2,2,1)$.}
\end{figure}

Continuing the example in Figure~\ref{fig:fold-hook}, see Figure~\ref{fig:fold-tower} for details of the map $\varphi_{(4,2,2,1)}$ on the permutation $841567392$. Again, the inverse descent set remains $\{2,3,7\}$ for each filling. The Macdonald weights for the left, middle and right partition shapes are $q^{4} t^{3}$, $q^{3} t^{3}$, and $q^{2} t^{5}$, respectively. Once again, the weights change with the bijection, but the set of permutations that determines the Macdonald polynomial is the same.

The inspiration for this paper comes from the generalized dual equivalence structures imposed on permutations for a given partition presented in \cite{Ass07,Ass15} and from the explicit transformations of these structures described in the unpublished preprint \cite{Ass-X}. The transformations, when applied to the dual equivalence structures for Macdonald polynomials, break generalized dual equivalence classes into smaller classes that are isomorphic to standard dual equivalence classes. The bijections in this paper are the direct translations of following these transformations in reverse. Therefore one can try to follow the algorithm for more general graphs in the hopes of resolving the following strengthening of Conjecture~\ref{conj:class}. This has been tested for partitions up to size $11$.

\begin{conjecture}
  Let $\mu = (\mu_1,\ldots,\mu_k,1^{m})$ and let $\nu = (\mu_1,\ldots,\mu_{k-1},\mu_k+1,1^{m-1})$, where $m>0$ and $\mu_k<\mu_{k-1}$. Then there exists a bijection $\varphi$ on permutations that preserves the inverse descent set such that if $u$ and $v$ are $\mu$-equivalent, then $\varphi(u)$ and $\varphi(v)$ are $\nu$-equivalent. In particular, for any partition $\mu$, we have
  \begin{equation}
    \widetilde{H}_{\mu}(X;q,t) = \sum_{\lambda} \left( \sum_{u \in \mathrm{SS}(\lambda)} q^{\inv_{\mu}(\varphi_{\mu}(u))} t^{\maj_{\mu}(\varphi_{\mu}(u))} \right) s_{\lambda}(X) ,
  \end{equation}
  where $\varphi_{\mu}$ is the (unique) sequence of bijections taking $(1^n)$ to $\mu$.
  \label{conj:bozo}
\end{conjecture}

%
\bibliographystyle{amsalpha} 
\bibliography{bozo.bib}
%

\end{document}